\documentclass[a4paper,10pt,leqno]{amsart}
\usepackage[total={6in,9in},
top=1in, left=1in, right=1in, bottom=1in]{geometry}
\usepackage{amsmath}
\usepackage[italian, english]{babel}
\usepackage{amsfonts}
\usepackage{amssymb}
\usepackage{dsfont}
\usepackage{mathrsfs}
\usepackage{graphicx}
\usepackage{setspace}
\usepackage{fancyhdr}
\usepackage{amsthm}
\usepackage{empheq}
\usepackage{cases}
\usepackage[all]{xy}
\usepackage{stmaryrd}
\usepackage{color}
\newcommand{\dist}{\operatorname{dist}}

\newcommand{\pa}[1]{{\left(#1\right)}}                  
\newcommand{\pair}[1]{g\left(#1\right)}      

\newcommand{\metric}{g}                          

\renewcommand{\hat}[1]{\widehat{#1}}
\renewcommand{\tilde}[1]{\widetilde{#1}}

\renewcommand{\l}{\lambda}

\newtheorem{theorem}{\textbf{Theorem}}[section]
\newtheorem{lemma}[theorem]{\textbf{Lemma}}

\newtheorem{defi}[theorem]{\textbf{Definition}}
\theoremstyle{remark}
\newtheorem{rem}[theorem]{\textbf{Remark}}

\newtheorem{exe}[theorem]{\textbf{Example}}
\numberwithin{equation}{section}

\title[]
{Elliptic and parabolic equations \\ with Dirichlet conditions at infinity \\ on Riemannian manifolds}

\date{\today} 

\keywords{Elliptic equations, parabolic equations, Dirichlet
conditions at infinity, Riemannian manifolds}

\subjclass[2010]{35J25, 35J67, 35K10, 35K20, 58J05, 58J32, 58J35}

\begin{document}

\maketitle
\begin{center}
\textsc{P. Mastrolia\footnote{Universit\`{a} degli Studi di Milano,
Italy. Email: paolo.mastrolia@gmail.com.}, D. D.
Monticelli\footnote{Politecnico di Milano, Italy. Email:
dario.monticelli@polimi.it.} and F. Punzo\footnote{Universit\`{a}
della Calabria, Italy. Email: fabio.punzo@unical.it. \\ P.
Mastrolia, D.D. Monticelli and F. Punzo are members of the Gruppo
Nazionale per l'Analisi Matematica, la Probabilit\`{a} e le loro
Applicazioni (GNAMPA) of the Istituto Nazionale di Alta Matematica
(INdAM). The research that led to the present paper was partially
supported by the grant ``Analisi Globale, PDE's e Strutture
Solitoniche'' of the group GNAMPA of INdAM.}}
\end{center}

\begin{abstract}
We investigate existence and uniqueness of bounded solutions of
parabolic equations with unbounded coefficients in $M\times \mathbb
R_+$, where $M$ is a complete noncompact Riemannian manifold. Under
specific assumptions, we establish existence of solutions satisfying
prescribed conditions at infinity, depending on the direction along
which infinity is approached. Moreover, the large-time behavior of
such solutions is studied. We consider also elliptic equations on
$M$ with similar conditions at infinity.
\end{abstract}

\bigskip

\section{Introduction}\setcounter{equation}{0}
We are concerned with bounded solutions of linear elliptic equations of the type
\begin{equation}\label{e1}
 a \Delta u \, + \, c u\,=\,f\quad \textrm{in}\;\; M\,,
\end{equation}
where $M$ is a complete, $m$--dimensional, noncompact Riemannian
manifold with metric $g$, $\Delta$ is the Laplace-Beltrami operator
with respect to $g$; furthermore, $a>0, c\leq 0,\, c, f \in
L^\infty(M)$. Observe that, at {\it infinity}, the function $a$ can
be unbounded, or it can tend to $0$, or it needs not to have a
limit.

Moreover, we study bounded solutions of linear parabolic Cauchy
problems of the following form
\begin{equation}\label{e2}
\left\{
\begin{array}{ll}
\,   \, \partial_t u = a \Delta u\, +\,  c u\, + f
&\textrm{in}\,\,S:=M\times (0, \infty),
\\& \\
\textrm{ }u \, = u_0& \textrm{in\ \ } M\times \{0\} \,,
\end{array}
\right.
\end{equation}
where $u_0\in L^\infty(M)$. Precise assumptions on $a, c\,, f$, and $u_0$ will be made in Section \ref{mr} below.

\smallskip

Existence and uniqueness of solutions of elliptic equations and of
parabolic problems have been largely investigated, in the case
$M=\mathbb R^m$ (see e.g. \cite{BK}, \cite{IKO}, \cite{KPT},
\cite{KPu1}, \cite{KPu2}, \cite{Mura1}, \cite{Mura2}, \cite{Pinch2},
\cite{Pinch}). In particular, in \cite{KPu1, KPu2} for suitable
classes of elliptic and parabolic equations, it is shown that it is
possible to prescribe Dirichlet type conditions at {\it infinity}. More
precisely, one can impose that the solutions at infinity, along
radial directions, approach any given continuous function defined on the unit
sphere $\mathbb S^{m-1}\subset \mathbb R^m$. It is also observed that in $\mathbb R^m$ such results
cannot hold in general for operators of the form appearing in
equation \eqref{e1} or in problem \eqref{e2}.

\smallskip

The situation is quite different on negatively curved Riemannian
manifolds. In fact, in \cite[Theorem 3.2]{Anderson} (see also
\cite{Anderson, AS, Choi, Sullivan}) it is shown that if $M$ is a
complete, simply connected Riemannian manifold with sectional
curvatures bounded between two negative constants, then for every
continuous function $\gamma$ on the sphere at infinity $S_\infty(M)$
there exists a unique solution $u$ of equation
\begin{equation}\label{ef11}
\Delta u = 0 \quad \textrm{in}\;\; M\,,
\end{equation}
that is equation \eqref{e1} with $c=f=0$,
such that $u=\gamma$ on $S_\infty(M)$.   In
this kind of results the Martin boundary plays a prominent role, and theoretical potential theory is heavily exploited. Indeed, more general elliptic
equations are considered, but always without any external forcing term $f$ and zero order term $c$.  Note that the presence of a zero order term in
equation \eqref{e1} may remarkably alter the situation. Indeed, let $\lambda>0$ be a constant; from \cite[Theorem 6.2]{Grig} it follows that equation
\begin{equation}\label{ef10}
\Delta u - \lambda u \,=\, 0\quad \textrm{in}\;\; M\,,
\end{equation}
that is equation \eqref{e1} with $f=0$ and $c=-\lambda$, admits a
unique bounded solution, if $M$ is the hyperbolic space $\mathbb
H^m$. Hence the result in \cite{Anderson}, which we recalled above,
does not hold for equation \eqref{ef10}, as it would imply
nonuniqueness of bounded solutions of equation \eqref{ef10}.

Some general results concerning conditions at infinity for solutions of parabolic equations on Riemannian manifold are established in \cite{Mura1, Mura2}. The Martin boundary is used; moreover, a representation formula is derived
for positive solutions of the Cauchy problem, associated to divergence form elliptic operators.

\smallskip

In this paper, under suitable assumptions on $a$ and $M$ (see {\bf
(HP1)} below), we prove existence of solutions of the elliptic
equation \eqref{e1} satisfying prescribed conditions at infinity.
More precisely, consider on $M$ the polar coordinates $(r,\theta)\in
(0,\infty)\times \mathbb S^{m-1}$, with respect to some fixed origin
$o\in M$; here $\mathbb S^{m-1}:=\{x\in \mathbb R^m\,:\, |x|=1\}$,
see also Section \ref{RG}. Define
\begin{equation}\label{e8}
\mathcal A:=\Big\{f\in C^\infty((0,\infty))\cap C^1([0,\infty)):
f'(0)=1,\, f(0)=0,\, f>0\text{ in } (0,\infty)\Big\}.
\end{equation}
We always make the following assumption:

\smallskip

\noindent {\bf (HP1).} $(i)$ There exist a point $o\in M$ with
$\operatorname{Cut}(o)=\emptyset$, i.e. with empty
\textit{cut-locus}, a constant $R_0>0$ and a function $\underline
a\in C([R_0, \infty))$ such that
\[ a(x) \geq \underline a (r)>0 \quad \textrm{for all}\,\, x\in M\,\, \textrm{with}\,\, r=\operatorname{dist}(x,o)\geq R_0\,;\]

$(ii)$   there exists a function $\psi\in \mathcal A$ such that
\[\operatorname{K}_\omega(x)\, \leq -\frac{\psi''(r)}{\psi(r)}\quad \textrm{for all}\;\; x=(r, \theta)\in M\setminus\{o\}\,,\]
\[\int_1^\infty \frac{dr}{\psi^{m-1}(r)}<\infty\,,\;\quad
\int_1^{\infty}\left(\int_{r}^\infty\frac{d\xi}{\psi^{m-1}(\xi)}\right)\frac
{\psi^{m-1}(r)}{\underline a(r)}\,dr \,<\, \infty\,.\] Here
$K_\omega(x)$ denotes the radial sectional curvature at $x$ (see
Section \ref{RG}).

\smallskip

From  {\bf (HP1)} it follows that in our result there is an interplay between the coefficient $a(x)$ and the manifold $M$, through the function $\psi$ which is in turn related to the radial sectional curvature. Observe that if $a(x)\equiv 1$, then condition {\bf (HP1)} implies
that $M$ is stochastically incomplete (see e.g. \cite{Grig}).
Moreover, it is direct to see that if {\bf (HP1)} holds, then $M$ is
non-parabolic, i.e. it admits a positive Green function
$G(x,y)<\infty$ for every $x,y\in M$, $x\neq y$; indeed, by
\cite[Theorem 4.2]{Grig2},
\[G(x,o) \leq \tilde C \int_{r(x)}^\infty\frac{d\xi}{\psi^{m-1}(\xi)}\quad (x\in M\setminus\{o\})\,.\]
Thus, we also have that for some compact subset $K\subset M$,
\[\int_{M\setminus K} \frac{G(x,o)}{a(x)}\,d\mu(x) <\infty\,.\]

\smallskip

Under suitable additional hypotheses on the coefficient $a(x)$ (see
conditions {\bf (HP0)} and \eqref{e13} below), we show that, for any
$\gamma\in C(\mathbb S^{m-1})$, there exists a unique solution of
the elliptic equation \eqref{e1} satisfying
\begin{equation}\label{e3}
\lim_{r\to \infty} u(r, \theta)=\gamma(\theta)\quad \textrm{uniformly w.r.t.}\;\; \theta\in \mathbb S^{m-1}\,.
\end{equation}
Note that condition \eqref{e3} can be regarded as a Dirichlet
condition at infinity, depending on the direction along which
infinity is approached.

Moreover, for any given function $\tilde \gamma\in C(\mathbb S^{m-1}\times[0, \infty))$, we prove that there exists a unique solution of problem
\eqref{e2} such that
\begin{equation}\label{e4}
\textrm{for each}\,\,T>0, \, \lim_{r\to \infty} u(r,\theta, t)=\tilde \gamma(\theta, t)\quad \textrm{uniformly w.r.t.}\;\; \theta\in \mathbb S^{m-1},
t\in [0,T]\,,
\end{equation}
provided
\begin{equation}\label{e5}
\lim_{r\to \infty} u_0(r, \theta)\,=\,\tilde \gamma(\theta, 0)\quad \textrm{uniformly w.r.t.}\;\; \theta\in \mathbb S^{m-1}\,.
\end{equation}
Note again that condition \eqref{e4} can be regarded as a
time-dependent Dirichlet condition at infinity, depending on the
direction along which infinity is approached.

\smallskip
We should note that when $\psi(r)=r$, and thus $M=\mathbb R^m$, our result cannot be applied (see Remark \ref{oss1f} below); this is in accordance with remarks made above (see \cite{KPu1}). On the other hand, we want to stress that our results are completely new also for problem
\begin{equation}\label{e2ef}
\left\{
\begin{array}{ll}
\,   \, \partial_t u =  \Delta u\
&\textrm{in}\,\,S
\\& \\
\textrm{ }u \, = u_0& \textrm{in\ \ } M\times \{0\} \,,
\end{array}
\right.
\end{equation}
i.e. problem \eqref{e2} with $a\equiv 1, c\equiv f\equiv 0$.

\smallskip
In order to obtain existence of solutions to problem \eqref{e1}
satisfying \eqref{e3}, we construct and use suitable {\it barrier
functions} at infinity (see Section \ref{SecBar} below).
Furthermore, by means of such barriers, we construct convenient
subsolutions and supersolutions, also depending on the time variable
$t$, in order to prescribe condition \eqref{e4} for solutions of
problem \eqref{e2}. We explicitly note that in order to construct
such barriers a prominent role is played by  {\bf (HP1)}. However,
the same existence results that we prove hold also on more general
Riemannian manifolds, if one a priori assumes the existence of such
barriers.

A similar approach has been used in \cite{Choi}, where
barriers which indeed are subharmonic functions have been exploited.
In fact, in \cite{Choi} it has been shown the existence of solutions
satisfying Dirichlet conditions at infinity only for equation \eqref{ef11}; moreover, it is supposed that $M$ is a spherically symmetric manifold with negative radial sectional curvature satisfying a suitable bound from above (see Section \ref{RG}).

\smallskip

The paper is organized as follows. In Section \ref{RG} we introduce
some basic notions and tools from Riemannian geometry, while in
Section \ref{mr} we state our main results, see Theorems \ref{teor1}
and \ref{teor2}. Section \ref{SecBar} is devoted to the construction
of suitable barrier functions at infinity, which are then used in
Sections \ref{dim1} and \ref{dim2} in the proofs of existence and
uniqueness of bounded solutions of problems \eqref{e1} and
\eqref{e2} with prescribed conditions at infinity. Finally, Section
\ref{SecEXE} contains some examples and applications of our main
theorems

\section{Preliminaries}\label{RG}
In this section we collect some notions and results from Riemannian Geometry following \cite{Grig, MRS}.

Let $\pa{M, \metric}$ be a Riemannian manifold of dimension $m$ with
metric $\metric$. Let $p\in M$ and let $(U, \varphi)$ be a
\emph{local chart} such that $p\in U$. Denote by $x^1,\ldots, x^m$,
$\,m=\operatorname{dim}\, M$ , the coordinate functions on $U$.
Then, at any $q\in U$ we have
\begin{equation}
\label{GP1.1} \metric=g_{ij}\,dx^{i} dx^{j}
\end{equation}
where $dx^{i}$ denotes the differential of the function $x^{i}$ and
$g_{ij}$ are the (local) components of the metric defined by
$g_{ij}=\pair{\frac{\partial}{\partial x^{i}},
\frac{\partial}{\partial x^{j}}}$. Its inverse will be denoted by
$g^{ij}$. In equation \eqref{GP1.1} and throughout this section we
adopt the Einstein summation convention over repeated indices.

Note that the Laplacian of a function $u\in C^2(M)$ has locally the
form
\begin{equation}\label{e27}
\Delta f= \mathfrak g^{-1}\frac{\partial}{\partial x_i}\left(\mathfrak g \, g^{ij} \frac{\partial f}{\partial x_j}\right),
\end{equation}
where
\[\mathfrak g=\sqrt{\rm{det}(g_{ij})}\,.\]

\smallskip

Now, fix a point $o\in M$ and denote by $\textrm{Cut}(o)$ the {\it
cut locus } of $o$. For any $x\in M\setminus
\big[\textrm{Cut}(o)\cup \{o\} \big]$, one can define the {\it polar
coordinates} with respect to $o$, see e.g. \cite{Grig}. Namely, for
any point $x\in M\setminus \big[\textrm{Cut}(o)\cup \{o\} \big]$
there correspond a polar radius $r(x) :=\operatorname{dist}(x, o)$
and a polar angle $\theta\in \mathbb S^{m-1}$ such that the shortest
geodesics from $o$ to $x$ starts at $o$ with direction $\theta$ in
the tangent space $T_oM$. Since we can identify $T_o M$ with
$\mathbb R^m$, $\theta$ can be regarded as a point of $\mathbb
S^{m-1}.$ For any $x_0\in M$ and for any $R>0$ we set $B_R(x_0):=
\big\{x\in M\,:\, \operatorname{dist}(x, x_0)<R\,\big\};$ in
addition, we denote by $d\mu$ the Riemannian volume element on $M$,
and by $S(x_0, R)$ the area of the sphere $\partial B_R(x_0)$.

The Riemannian metric in $M\setminus\big[\textrm{Cut}(o)\cup \{o\} \big]$ in polar coordinates reads
\[g= dr^2+A_{ij}(r, \theta)d\theta^i d\theta^j, \]
where $(\theta^1, \ldots, \theta^{m-1})$ are coordinates in $\mathbb S^{m-1}$ and $(A_{ij})$ is a positive definite matrix. Let $(A^{ij})$ denote the
inverse matrix of $(A_{ij})$. It is not difficult to see that the Laplace-Beltrami operator in polar coordinates has the form
\begin{equation}\label{e6} \Delta = \frac{\partial^2}{\partial r^2} +
\mathcal F(r, \theta)\frac{\partial}{\partial r}+\Delta_{S_{r}},
\end{equation}
where $\mathcal F(r, \theta):=\frac{\partial}{\partial
r}\big(\log\sqrt{A(r,\theta)}\big)$, $A(r,\theta):=\det
(A_{ij}(r,\theta))$, $\Delta_{S_r}$ is the Laplace-Beltrami operator
on the submanifold $S_{r}:=\partial B_r(o)\setminus
\textrm{Cut}(o)$\,.

$M$ is a {\it manifold with a pole}, if it has a point $o\in M$ with $\textrm{Cut}(o)=\emptyset$. The point $o$ is called {\it pole} and the polar
coordinates $(r,\theta)$ are defined in $M\setminus\{o\}$.

A manifold with a pole is a {\it spherically symmetric manifold} or a {\it model}, if the Riemannian metric is given by
\begin{equation}\label{e7}
g= dr^2+\psi^2(r)d\theta^2,
\end{equation}
where $d\theta^2=\beta_{ij}d\theta^i d \theta^j$ is the standard
metric in $\mathbb S^{m-1}$, $\beta_{ij}$ being smooth functions of
$\theta^1, \ldots, \theta^{m-1},$ and $\psi\in \mathcal A$, with
being defined in \eqref{e8}. In this case, we write $M\equiv
M_\psi$; furthermore, we have $\sqrt{A(r,\theta)}=\psi^{m-1}(r)$, so
that
\[\Delta = \frac{\partial^2}{\partial r^2}+ (m-1)\frac{\psi'}{\psi}\frac{\partial}{\partial r}+ \frac1{\psi^2}\Delta_{\mathbb S^{m-1}}\,,\]
where $\Delta_{\mathbb S^{m-1}}$ is the Laplace-Beltrami operator in
$\mathbb S^{m-1}\,.$ In addition, the boundary area of the geodesic
sphere $\partial S_R$ is computed by
\[S(o, R)=\omega_m\psi^{m-1}(R),\]
$\omega_m$ being the area of the unit sphere in $\mathbb R^m$. Also, the volume of the ball $B_R(o)$ is given by
\[\mu(B_R(o))=\int_0^R S(o, \xi)\,d\xi\,. \]

Observe that for $\psi(r)=r$, $M=\mathbb R^m$, while for $\psi(r)=\sinh r$, $M$ is the $m-$dimensional hyperbolic space $\mathbb H^m$.

Let us recall comparison results for sectional and Ricci curvatures
that will be used in the sequel. Let
$\operatorname{Cut}^*(o)=\operatorname{Cut}(o)\cup\{o\}$ and, for
any $x\in M\setminus\operatorname{Cut}^*(o)$, denote by
$\textrm{Ric}_o(x)$ the {\it Ricci curvature} at $x$ in the
direction $\frac{\partial}{\partial r}$. Let $\omega$ denote any
pair of tangent vectors from $T_xM$ having the form
$\left(\frac{\partial}{\partial r} ,X\right)$, where $X$ is a unit
vector orthogonal to $\frac{\partial}{\partial r}$. Denote by
$\textrm{K}_{\omega}(x)$ the {\it sectional curvature} at the point
$x$ of the $2$-section determined by $\omega$. Observe that (see
\cite[Section 15]{Grig}, \cite{Ichi1}, \cite{Ichi2}), if
$\operatorname{Cut}(o)=\emptyset$ and
\begin{equation}\label{e9}
\textrm{K}_{\omega}(x)\leq -\frac{\psi''(r)}{\psi(r)}\quad \textrm{for all}\;\; x=(r,\theta)\in M\setminus\{o\},
\end{equation}
for some function $\psi\in \mathcal A$, then
\begin{equation}\label{e10}
\mathcal F(r, \theta)\geq (m-1)\frac{\psi'(r)}{\psi(r)}\quad
\textrm{for all}\;\; r>0,\, \theta \in \mathbb S^{m-1}\,.
\end{equation}

On the other hand, if
\begin{equation}\label{e11}
\textrm{Ric}_{o}(x)\geq -(m-1)\frac{\phi''(r)}{\phi(r)}\quad \textrm{for all}\;\; x=(r,\theta)\in M\setminus\operatorname{Cut}^*(o),
\end{equation} for some function $\phi\in \mathcal A$, then
\begin{equation}\label{e12}
\mathcal F(r, \theta)\leq (m-1)\frac{\phi'(r)}{ \phi(r)}\quad \textrm{for all}\;\; r>0, \theta \in \mathbb S^{m-1}\,\,\textrm{with}\,\, x=(r,
\theta)\in M\setminus \operatorname{Cut}^*(o)\,.
\end{equation}

Note that if $M_\psi$ is a model manifold, then for any $x=(r,
\theta)\in M_\psi\setminus\{o\}$
\[\textrm{K}_{\omega}(x)=-\frac{\psi''(r)}{\psi(r)},\]
and
\[\textrm{Ric}_{o}(x)=-(m-1)\frac{\psi''(r)}{\psi(r)}\,.\]

Recall that a Riemannian manifold $M$ is said to be {\it non-parabolic} if it admits a nonconstant positive superharmonic function, and {\it
parabolic} otherwise (see e.g. \cite{Grig}). Observe that $M$ is non-parabolic if and only if it admits a positive Green function $G(x,y)<\infty$ for
every $x,y \in M, x \neq y;$ moreover,
\begin{equation*}\label{e12a}
\int_1^{\infty}\frac{d\xi}{S(o, \xi)}=\infty,
\end{equation*}
for some $o\in M$, if and only if $M$ is parabolic (see \cite[Theorem 7.5, Corollary 15.2]{Grig}).

In the sequel, we also consider {\it Cartan-Hadamard} Riemannian manifolds, i.e. simply connected complete noncompact Riemannian manifolds with
nonpositive sectional curvatures. Observe that (see, e.g. \cite{Grig}, \cite{Grig3}) on Cartan-Hadamard manifolds we have
$\operatorname{Cut}(o)=\emptyset$ for any $o \in M.$

\section{Existence and uniqueness results}\label{mr}
\setcounter{equation}{0} Before stating our main results, we need some preliminary materials. Concerning the coefficients of the operator $\mathcal
L$, $c$ and $f$ we make the following set of assumptions:
\[
\textrm{\ \ } \left\{
\begin{array}{l}
\textrm{(i)} \quad \; a\in C^{\sigma}_{\textrm{loc}}(M)\;\,\textrm{for some}\;\, \sigma\in (0,1),\; a>0 \,\, \textrm{in}\;\; M;
\\ \textrm{(ii)} \;\;\;\,  c, f\in C_{\textrm{loc}}^{\sigma}(M)\cap L^\infty(M)\,.
\end{array}
\right. \leqno(\textrm{\bf{HP0}}) \] Note that the coefficient $a$ can be unbounded at infinity.

\medskip
For any $R>0, \delta>0, \theta_0\in \mathbb S^{m-1}$ set
\begin{equation}\label{e510}
\mathcal C^R_{\theta_0, \delta}:=\Big\{ x\equiv (r, \theta)\in
M\,:\, r>R, \,\dist_{\mathbb S^{m-1}}(\theta, \theta_0)
<\delta\,\Big\}\,,\end{equation} where $\dist_{\mathbb
S^{m-1}}(\theta, \theta_0)$ denotes the geodesic distance on
$\mathbb S^{m-1}$ between $\theta$ and $\theta_0$.

\medskip

Subsolutions, supersolutions and solutions of equation \eqref{e1} and of problem \eqref{e2} are meant as follows.

\begin{defi}
A function $u\in C^2(M)$ is a {\em subsolution} of equation \eqref{e1} if
$$  a\Delta u(x)\,+\, c(x) u(x)\, \geq \,
f(x)\quad \textrm{for any}\,\;\; x\in M\,.$$ A {\em supersolution} is defined replacing the previous ``$\,\geq \,$'' with ``$\,\leq\,$''. Finally, a
{\em solution} is both a subsolution and a supersolution.
\end{defi}

\begin{defi}
A function $u\in C_{x,t}^{2,1}(M\times (0, \infty))\cap C(M\times
[0, \infty))$ is a {\em subsolution} of problem \eqref{e2} if
\begin{itemize}
\item[(i)] $\partial_t u(x,t) \leq a\Delta u(x,t)\, +\,  c(x) u(x,t)\, + f(x)\quad \textrm{for any}\,\;\; x\in M,\, t\in (0, \infty)$,\;
\item[(ii)]
$u(x,0)\leq u_0(x)\quad \textrm{for any}\;\, x\in M\,.$
\end{itemize}
A {\em supersolution} is defined replacing the previous two ``$\,\leq\,$'' with ``$\,\geq\,$''. Finally, a {\em solution} is both a subsolution and a
supersolution.
\end{defi}

In the following, the function
\[\omega(r):= \max_{i,j=1, \ldots, m-1,\\\
\theta\in \mathbb S^{m-1}} \left\{\left|\frac{\partial A^{ij}(r,
\theta)}{\partial \theta^i}\right|+\frac 1 2 \frac{|A^{ij}(r,
\theta)|}{A(r, \theta)}\left|\frac{\partial A(r, \theta)}{\partial
\theta^i}\right| + |A^{ij}(r, \theta)|\right\} \quad (r>0)
\]
will play an important role.

Our first result concerns the existence and uniqueness of solutions of elliptic equations with prescribed conditions at infinity.
\begin{theorem}\label{teor1}
Let assumptions {\bf (HP0)}-{\bf (HP1)} be satisfied. Let $\gamma\in
C(\mathbb S^{m-1})$ and $c\leq 0$. Suppose that, for some $C_0>0$,
$R_0>0$
\begin{equation}\label{e13}
\frac 1{\underline a(r)} \geq C_0 \, \omega(r) \quad \textrm{for
all}\;\; x=(r, \theta)\in M,\, r\geq R_0\,.
\end{equation}
Then there exists a unique solution of equation \eqref{e1} such that condition \eqref{e2} is satisfied.
\end{theorem}

\smallskip


Moreover, concerning the parabolic problem \eqref{e2}, we have the
following result.
\begin{theorem}\label{teor2}
Let assumptions {\bf (HP0)}-{\bf (HP1)} be satisfied. Let $\tilde
\gamma\in C(\mathbb S^{m-1}\times[0,\infty))$, and $u_0\in C(M)\cap
L^\infty(M)$. Suppose that conditions \eqref{e5} and \eqref{e13} are
satisfied. Then there exists a unique solution of problem \eqref{e2}
such that condition \eqref{e4} is satisfied.
\end{theorem}

\begin{rem}\label{rem1}
Note that if $M\equiv M_\psi$ is a model, then $\omega(r)=\frac
1{\psi^2(r)}$. So, condition \eqref{e13} reads as follows
\begin{equation}\label{e13a}
\underline a(r) \leq \frac{\psi^2(r)}{C_0}   \quad \textrm{for
all}\;\, x=(r, \theta)\in M,\, r\geq R_0\,.
\end{equation}
\end{rem}

\begin{rem}\label{oss1f}
Note that if $\psi(r)=r$, and thus $M=\mathbb R^m$, conditions {\bf
(HP1)} and \eqref{e13a} cannot be simultaneously satisfied. Hence,
our results cannot be applied.
\end{rem}

\section{Construction of barriers at infinity}\label{SecBar}

\begin{lemma}\label{lemma3}
Let assumptions {\bf (HP0)}-{\bf (HP1)} be satisfied. Then there
exists a supersolution $V$ of equation
\begin{equation}\label{e14}
 a(x) \Delta\, V\, =\, -\, 1\quad\textrm{in}\;\;M\,,
\end{equation}
such that
\begin{equation}\label{e16}
V(x)>0 \quad \textrm{for all}\;\, x\in M\,,
\end{equation}
and
\begin{equation}\label{e15}
\lim_{r(x)\to \infty}\, V(x)\,=\,0\,.
\end{equation}
\end{lemma}
\begin{proof}
Define
$$a_0(r):=\left\{
\begin{array}{ll}
 \,   \,  \frac 1{ \bar C \underline a(R_0)} &\textrm{if}\,\,r\in [0, R_0)
\\&\\
\textrm{ }\frac 1{\bar C \underline a(r)} & \textrm{if}\,\,  r\in [R_0,\infty)\,;
\end{array}
\right. $$ here
$$\bar C:=\frac 1{\underline a(R_0)}\min\left\{\min_{\overline B_{R_0}} a, \underline a(R_0)\right\}\in (0, 1]\,. $$

Clearly, $a_0\in C([0,\infty))$; moreover, by assumption {\bf
(HP1)}-(i) and by the definition of $a_0(r)$ and $\bar{C}$,
\begin{equation}\label{e17}
a(x) \geq \frac 1{  a_0(r(x))}\quad\textrm{for every}\;\; x\in M\,.
\end{equation}

Note that for every $r>0$
\begin{equation}\label{e18}
\begin{aligned}
&\left(\int_{r}^\infty\frac{d\xi}{\psi^{m-1}(\xi)}\right)\left(\int_0^{r}
a_0(t)\psi^{m-1}(t) dt\right) -
\int_0^{r}\left(\int_t^\infty\frac{d\xi}{\psi^{m-1}(\xi)}\right)
 a_0(t)\psi^{m-1}(t)\, dt \\
&\qquad \le\int_{r}^\infty\frac{d\xi}{\psi^{m-1}(\xi)}\int_0^{r}
\frac{\int_t^\infty\frac{d\xi}{\psi^{m-1}(\xi)}}{\int_{r}^\infty
\frac{d\xi}{\psi^{m-1}(\xi)}}
 a_0(t)\psi^{m-1}(t) dt  - \int_0^{r}\int_t^\infty\frac{d\xi}{\psi^{m-1}(\xi)} a_0(t)\psi^{m-1}(t) dt= 0.
\end{aligned}
\end{equation}
From \eqref{e18} and hypothesis {\bf (HP1)} we get
$$H:=\limsup_{\rho\to \infty}\left\{\int_{\rho}^\infty\frac{d\xi}{\psi^{m-1}(\xi)}\int_0^{\rho}  a_0(t)\psi^{m-1}(t) dt - \int_0^{\rho}\int_t^\infty\frac{d\xi}
{\psi^{m-1}(\xi)} a_0(t)\psi^{m-1}(t) dt\right\} \le 0\,.$$

Define for every $x\in M$
\begin{equation}\label{e19}
\begin{aligned}
&V(x)\equiv V(r(x))\\
&\,\,\,:=\left(\int_{r(x)}^\infty\frac{d\xi}{\psi^{m-1}(\xi)}\right)\left(\int_0^{r(x)}
a_0(t)\psi^{m-1}(t)\, dt\right) -
\int_0^{r(x)}\left(\int_t^\infty\frac{d\xi}{\psi^{m-1}(\xi)}\right)
a_0(t)\psi^{m-1}(t) dt - H\,.
\end{aligned}
\end{equation}
We have that $V\in C^2(M)$. Furthermore, for every $r>0$
\begin{align}
\label{e20} V'(r)&=\,-\frac1{\psi^{m-1}(r)}\int_0^{r} a_0(t)
\psi^{m-1}(t) dt < 0,\\
\label{e21} V''(r)&=\,(m-1)\frac{\psi'(r)}{\psi^{m}(r)}\int_0^r
a_0(t) \psi^{m-1}(t) dt- a_0(r)\;.
\end{align}
Then $V'(0)=0$, $V''(0)=-\frac{a_0(0)}{m}<0$ and in view of
\eqref{e21}, {\bf (HP1)}, \eqref{e10}, \eqref{e6} and \eqref{e17} we
obtain
\begin{equation}\label{e22}
\begin{aligned}
a(x) \Delta V(x)& = a(x) \left[V''(r)+\mathcal F(r, \theta) V'(r) \right] \\ & \le a(x)\left[V''(r)+(m-1)\frac{\psi'(r)}{\psi(r)} V'(r)\right]\\
&\leq - a(x) a_0(r(x))\leq -1 \quad \textrm{for all}\;\; x\in M.
\end{aligned}
\end{equation}
Finally, it is easily checked that \eqref{e16} and \eqref{e15} are satisfied. This completes the proof.
\end{proof}

\bigskip

\begin{lemma}\label{lemma4}
Let assumptions {\bf (HP0)}-{\bf (HP1)} be satisfied. Suppose that
condition \eqref{e13} holds. Let $V$ be defined as in \eqref{e19}.
Then there exist $\hat R>0, \hat \delta>0, \hat C>0$ such that for
any $\theta_0\in \mathbb S^{m-1}$ the function
\begin{equation}\label{e23}
h(x; \theta_0):= \hat C V(r) +  \operatorname{dist}_{\mathbb S^{m-1}}^2(\theta, \theta_0)\quad (x=(r, \theta)\in M)
\end{equation} is a supersolution of equation
\begin{equation}\label{e23}
a(x)\Delta h(\cdot\, ; \theta_0)\, = - 1 \quad \textrm{in}\;\;
\mathcal C^{\hat R}_{\theta_0, \hat \delta}\,.
\end{equation}
Moreover,
\begin{equation}\label{e24}
h(x; \theta_0)>0 \quad \textrm{for all} \;\;  x\in \overline{\mathcal C^{\hat R}_{\theta_0, \hat\delta}}\,\;,
\end{equation}
and for any $0<\delta \leq \hat \delta$, $R\geq\hat{R}$ there exists
$m_{\delta,R}>0$ independent of $\theta_0$ such that
\begin{equation}\label{e25}
h(x;\theta_0)\,\geq m_{\delta,R}>\,0 \quad \textrm{for all} \;\,
x\in\partial\mathcal C^R_{\theta_0, \delta}\,.
\end{equation}
In addition,
\begin{equation}\label{e26}
\lim_{r\to \infty} h(r, \theta_0;\theta_0)\,=\,0 \quad
\textrm{uniformly w.r.t.}\,\, \theta_0\in \mathbb S^{m-1}\,.
\end{equation}
\end{lemma}
\begin{proof}
In view of \eqref{e27} and \eqref{e13}, for each $r>0, \theta_0\in \mathbb S^{m-1}$ we have that
\begin{equation}\label{e28}
\begin{aligned}
\Delta_{S_r} \operatorname{dist}_{\mathbb S^{m-1}}^2(\theta,
\theta_0) &= \frac 1{\sqrt{A(r,
\theta)}}\frac{\partial\big[\sqrt{A(r, \theta)} A^{ij}(r,
\theta)\big]}{\partial
\theta^i}\frac{\partial\operatorname{dist}_{\mathbb
S^{m-1}}^2(\theta, \theta_0) }{\partial \theta^j} + A^{ij}(r,
\theta)\frac{\partial^2 \operatorname{dist}_{\mathbb
S^{m-1}}^2(\theta, \theta_0)}{\partial \theta^i \partial
\theta^j}\\& \leq C \max_{i,j=1, \ldots, m-1,\ \theta\in \mathbb
S^{m-1}} \left\{\left|\frac{\partial A^{ij}(r, \theta)}{\partial
\theta^i}\right|+\frac 1 2 \frac{|A^{ij}(r, \theta)|}{A(r,
\theta)}\left|\frac{\partial A(r, \theta)}{\partial \theta^i}\right|
+ |A^{ij}(r, \theta)|\right\}\\ &\leq \frac
C{C_0}\frac{1}{\underline a(r)}\,\,\textrm{whenever}\,\,\, \theta\in
\mathbb S^{m-1},\, \operatorname{dist}_{\mathbb S^{m-1}}(\theta,
\theta_0)<\hat \delta\,,
\end{aligned}
\end{equation}
for some positive constants $C, \hat \delta$ independent of $r, \theta, \theta_0.$ From \eqref{e17}, \eqref{e22}, and \eqref{e28} we deduce that
\begin{equation}\label{e29}
\begin{aligned}
a(x) \Delta h(x; \theta_0) &\leq a(x)\left[- \hat C a_0(r(x)) + \frac{C}{C_0}\frac 1{\underline a(r(x))}\right]\\
& \leq  a(x)\left(-\hat C +\frac{C}{C_0}\right)a_0(r(x)) \leq -1 \quad \textrm{for all}\,\; x\in \mathcal C^{\hat R}_{\theta_0, \hat \delta}\,,
\end{aligned}
\end{equation}
provided $\hat C\geq \frac C{C_0}+1$ and $\hat{R}>R_0$. Hence
\eqref{e23} has been shown. Finally, \eqref{e24}, \eqref{e25},
\eqref{e26} follow by the very definition of $h$, with
$m_\delta=\min\{\delta^2,\hat{C}V(R)\}$.
\end{proof}

\section{Proof of Theorem \ref{teor1}}\label{dim1}
\begin{proof}[Proof of Theorem \ref{teor1}]
By standard results (see, e.g. \cite{GT}), for any $j\in \mathbb N$ there exists a unique classical solution $u_j\in C^{2}(B_j)\cap
C(\overline{B_j})$ of problem
\begin{equation}\label{e30}
\left\{
\begin{array}{ll}
\,  a\Delta u_j + c u_j = f  &\textrm{in}\,\,B_j
\\& \\
\textrm{ }u_j \, = \gamma_1 & \textrm{on\ \ } \partial B_j \,,
\end{array}
\right.
\end{equation}
where $$\gamma_1(x)\equiv  \gamma_1(r, \theta):=\gamma(\theta) \quad \textrm{for all} \;\; r>0, \theta\in \mathbb S^{m-1}\,.$$

\smallskip

Let $V$ be the supersolution provided by Lemma \ref{lemma3} and
consider $W=V+1$. Since $W\geq1$ on $M$, we have that for any $j\in
\mathbb N$ the function
\[\overline u:= \bar C W\]
is a supersolution to problem \eqref{e30}, provided
\[\bar C\geq \max\{\|f\|_{\infty}, \|\gamma\|_{\infty}\}\,.\]
Analogously, we have that for any $j\in \mathbb N$, the function $\underline u:= -\overline u$ is a subsolution to the same problem. By the
comparison principle, for any $j\in \mathbb N$,
\begin{equation}\label{e32}
|u_j|\leq \bar C \|W \|_{\infty}=: \tilde C\quad \textrm{in}\;\;
B_j\,.
\end{equation}
By usual compactness arguments (see, e.g., \cite{GT}), there exists a subsequence $\{u_{j_k}\}\subset \{u_j\}$ and a function $u\in C^{2}(M)$ such
that
\[u:=\lim_{k\to \infty} u_{j_k}\quad \textrm{in}\;\; M\,.\]
Moreover, $u$ solves equations \eqref{e1}. In the sequel, we still denote by $\{u_j\}$ the sequence $\{u_{j_k}\}$.

\smallskip

Now, we show that \eqref{e3} holds. In order to do this, fix any
$\theta_0\in \mathbb S^{m-1}.$ For any $R>0, \delta>0, j>R$ define
\[ N^j_\delta\,\equiv\,N^{j,  R}_{\delta, \theta_0} \,:=\, \mathcal C^{R}_{\theta_0,\delta} \cap B_j\,,\]
We shall prove some estimates in $N^j_\delta$, by constructing suitable supersolutions and subsolutions to problem
\begin{equation}\label{e33}
\left\{
\begin{array}{ll}
\,   a\Delta v + c v= 0  &\textrm{in}\,\,N^{j,  R}_{\delta,
\theta_0}
\\& \\
\textrm{ } v \, = 0 & \textrm{on\ \ } \partial N^{j,  R}_{\delta, \theta_0} \,,
\end{array}
\right.
\end{equation}
and then using the comparison principle.

\smallskip
Let $\hat R$ and $\hat\delta$ be given by Lemma \ref{lemma4}. Fix
any $0<\epsilon<1$. Define, for some $K>0$ to be fixed later,
\[\overline v(x):= K h(x; \theta_0) + \gamma(\theta_0)+\epsilon - u_j(x)\quad \big(x\in \overline{N^j_{\hat\delta}}\big)\,,\]
where for every $\delta\in(0,\hat\delta]$ we set $N^j_{\delta}\equiv
N^{j,  \hat R}_{\delta, \theta_0}$. Note that there exists
$0<\delta=\delta(\epsilon)\leq\hat \delta$ such that for all
$j\in\mathbb N$, $x\equiv(r,\theta)\in
\partial N^j_\delta\cap
\partial B_j$
\begin{equation}\label{e34}
|u_j(x) - \gamma(\theta_0)| =|\gamma(\theta)-\gamma(\theta_0)|<\epsilon\,.
\end{equation}
Observe that, since $\gamma\in C(\mathbb S^{m-1})$ and $\mathbb S^{m-1}$ is compact, such a $\delta=\delta(\epsilon)$ does not depend on $\theta_0$.
From \eqref{e34} we obtain
\begin{equation}\label{e35}
\overline v(x) \geq 0 \quad \textrm{for all}\;\; x\in \partial
N^j_\delta,\, r(x)=j\,.
\end{equation}

From \eqref{e32} and \eqref{e25} we get that
\begin{equation}\label{e36}
\overline v(x)\geq K m_{\delta,\hat{R}} + \gamma(\theta_0) +\epsilon
- \tilde C \geq 0\quad \textrm{for all}\;\; x\in \partial
N^j_\delta,\, \hat R<r(x)<j,
\end{equation}
choosing
\begin{equation}\label{e37}
K \geq \frac{\|\gamma \|_{\infty}+ \tilde C}{m_{\delta,\hat{R}}}\,,
\end{equation}
where $\tilde C$ is defined by \eqref{e32}; hence, $K$ also depends
on $\delta(\epsilon)$. From \eqref{e32} and \eqref{e25} we can infer
that, if \eqref{e37} holds, then
\begin{equation}\label{e38}
\overline v(x) \geq 0\quad \textrm{for all}\;\; x\in \partial
N^j_\delta,\, r(x)=\hat R\,.
\end{equation}
Moreover, for all $x\in N^j_\delta$
\begin{equation}\label{e39}
a\Delta  \overline v + c \overline v \leq - K + c K h +
\|c\|_{\infty}(\|\gamma\|_{\infty}+ 1)-f\leq 0\,,
\end{equation}
for
\begin{equation}\label{e40}
K\geq \|c\|_{\infty}(\|\gamma \|_{\infty}+ 1)+\|f\|_\infty\,.
\end{equation}
Now, choose $\delta>0$ so small that \eqref{e34} holds and $K$ so
that \eqref{e37} and \eqref{e40} hold. Hence, from \eqref{e35},
\eqref{e36}, \eqref{e38} and \eqref{e39}, the function $\overline v$
is a supersolution of problem \eqref{e33}. So, by the comparison
principle,
\[\overline v \geq 0 \quad \textrm{in}\;\; N^j_\delta\,.\]
Therefore,
\begin{equation}\label{e41}
u_j \leq K h + \gamma(\theta_0)+\epsilon \quad \textrm{in}\;\; N^j_\delta\,.
\end{equation}

Similarly we can show that, for the same
$\delta=\delta(\varepsilon)$ as in the previous calculations (see
\eqref{e34}), the function
\[\underline v(x):= - K h(x;\theta_0) + \gamma(\theta_0) - \epsilon - u_j(x)\quad \big(x\in \overline{N^j_\delta}\big)\,\]
is a subsolution to problem \eqref{e33}. Hence, by the comparison principle,
\[\underline v \leq 0 \quad \textrm{in}\;\; N^j_\delta\,.\]
Therefore,
\begin{equation}\label{e42}
u_j \geq - K h + \gamma(\theta_0) - \epsilon \quad \textrm{in}\;\;
N^j_\delta\,.
\end{equation}
Letting $j\to \infty$ in \eqref{e41} and \eqref{e42} we obtain
\[-K h(x;\theta_0) -\epsilon \leq u(x) - \gamma(\theta_0)\leq K h(x;\theta_0) + \epsilon \quad \textrm{for any}\;\; x\equiv(r,\theta)
\in\mathcal C^{\hat R}_{\theta_0, \delta}\,.\] In view of
\eqref{e26}, taking $\theta=\theta_0$, letting $r\to \infty$, and
$\epsilon\to 0^+$, we obtain that $u(r,\theta_0)\to\gamma(\theta_0)$
uniformly for $\theta_0\in \mathbb S^{m-1}$.

It remains to prove the uniqueness of the solution. To do this, suppose, by contradiction, that there exist two solutions $u_1$ and $u_2$ of equation
\eqref{e1} satisfying condition \eqref{e2}. So, $w:= u_1-u_2$ solves equation
\begin{equation}\label{e43}
a\Delta w \,+\, c w \,=\,0\quad \textrm{in}\;\;\, M\,;
\end{equation}
moreover,
\begin{equation}\label{e44}
\lim_{r(x)\to \infty} w(x)=0\,.
\end{equation}
Fix any $\epsilon>0$. In view of \eqref{e44}, there exists $R_\epsilon>0$ such that for any $R>R_\epsilon$
\begin{equation}\label{e45}
w \leq \epsilon\quad \textrm{on}\;\; \partial B_R\,.
\end{equation}
Thus, $u$ is a subsolution to problem
\[
\left\{
\begin{array}{ll}
\,   a\Delta w + c w \,=\, 0&\textrm{in}\,\, B_R
\\& \\
\textrm{ } w \, = \epsilon & \textrm{on\ \ } \partial B_R\,.
\end{array}
\right.
\]
By the comparison principle,
\[  w \leq \epsilon \quad \textrm{in}\;\; B_R\,. \]
Letting $R\to\infty$ and $\epsilon \to 0^+$, we have
\[ w \leq 0 \quad \textrm{in}\;\; M\,.\]
Similarly, it can be shown that $w\geq 0$ in $M$. So, $u_1 \equiv u_2.$ This completes the proof of Theorem \ref{teor1}.
\end{proof}

\section{Proof of Theorems \ref{teor2}}\label{dim2}
Here and in the following, $\{\zeta_j\}\subset C^\infty_c(B_j)$ will
be a sequence of functions such that, for each $j\in \mathbb N$,
$0\leq \zeta_j\leq 1$, $\zeta_j\equiv 1$ in $B_{j/2}\,.$

\medskip

\begin{proof}[Proof of Theorem \ref{teor2}] Fix any $T>0$. For any $j\in \mathbb N$ let $u_j\in C^{2,1}_{x,t}({B}_j\times (0,T])\cap
C(\overline{B}_j\times[0, T])$ be the unique solution (see, e.g., \cite{LSU}) of problem
\begin{equation}\label{e46}
\left\{
\begin{array}{ll}
\,   \partial_t u_j = a\Delta u_j + c u_j + f
&\textrm{in}\,\,B_j\times (0,T],
\\& \\
\textrm{ }u_j \, = \tilde \gamma_1 &\textrm{in\ \ } \partial B_j\times (0, T] \,,
\\& \\
\textrm{ } u_j \,= u_{0,j} &\textrm{in \ \ } B_j\times \{0\}\,,
\end{array}
\right.
\end{equation}
where $$\tilde \gamma_1(x,t)\equiv \tilde \gamma_1(r, \theta, t)=:\tilde \gamma(\theta,t)\quad \textrm{for all}\;\; r>0, \theta\in \mathbb S^{m-1},
t\in [0,T]\,;$$ furthermore,
\begin{equation}\label{e47}
u_{0,j}(x):=\zeta_j(x) u_0(x) + [1-\zeta_j(x)]\tilde \gamma_1(x,0)\quad \textrm{for all}\;\; x\in\overline{B}_j\,.
\end{equation}

\smallskip
It is easily seen that the function
\[\overline{v}(x,t):= C e^{\beta t} \quad \big((x,t)\in M\times [0,T]\big)\]
is a supersolution of problem \eqref{e46} for any $j\in \mathbb N$, provided that
\[\beta\geq
1+\|c \|_\infty,\;\; C \geq \max\{\|f\|_\infty,
\|\tilde\gamma\|_\infty, \|u_0\|_\infty\}\,.\] Thus, by the
comparison principle,
\begin{equation}\label{e48}
u_j(x,t)\leq \bar v(x,t) \quad \textrm{for all}\;\; (x,t)\in M\times [0,T]\,.
\end{equation}
Furthermore, the function
\[\underline{v}(x,t):= - C e^{\beta t} \quad \big((x,t)\in M\times [0,T]\big)\]
is a subsolution of problem \eqref{e46} for any $j\in \mathbb N$. Thus, by the comparison principle,
\begin{equation}\label{e49}
u_j(x,t)\geq \underline{v}(x,t) \quad \textrm{for all}\;\; (x,t)\in M\times [0,T]\,.
\end{equation}
From \eqref{e48}-\eqref{e49} we obtain
\begin{equation}\label{e50}
\big|u_j(x,t)\big|\leq C e^{\beta  T}=: K_T \quad \textrm{for all}\;\; (x,t)\in M\times [0,T]\,.
\end{equation}
By usual compactness arguments (see, e.g., \cite{LSU}), there exists
a subsequence $\{u_{j_k}\}\subseteq \{u_j\}$ which converges, as
$k\to \infty,$ to a solution $u\in C^2(M\times(0,T ])\cap
C(M\times[0,T])$ of problem \eqref{e2}.

\medskip

We claim that \eqref{e4} holds. In fact, fix any $\theta_0\in
\mathbb S^{m-1}$, $t_0\in [0, T],$ and $0<\epsilon<1$. Let $\hat R$
and $\hat \delta$ be defined as in Lemma \ref{lemma4}. Let \[
t_\delta:=\max\{t_0-\delta, 0\}\, \quad \textrm{for
any}\,\,0<\delta\leq \hat \delta\,. \] Then there exists a positive
constant $0<\delta=\delta(\epsilon)<\hat \delta$ such that
\begin{equation}\label{e51}
\tilde \gamma(\theta_0, t_0)-\epsilon \leq \tilde \gamma(\theta,
t)\leq \tilde \gamma(\theta_0, t_0)+\epsilon\quad
\textrm{whenever}\;\, \dist_{\mathbb
S^{m-1}}(\theta,\theta_0)<\delta,\, t\in [t_\delta, t_0]\,.
\end{equation}
Note that $\tilde \gamma$ is continuous in the compact set $\mathbb S^{m-1}\times [0,T]$, thus such a $\delta=\delta(\epsilon)$ does not depend on
$\theta_0$ and $t_0$. Furthermore, due to \eqref{e5}, there exists $R_\epsilon>0$ such that
\begin{equation}\label{e51a}
\tilde\gamma(\theta,0)-\epsilon \leq u_0(r,\theta) \leq
\tilde\gamma(\theta, 0) +\epsilon \quad \textrm{for all}\;\;
x=(r,\theta)\in M\setminus B_{R_\epsilon}\,.
\end{equation}

Consider the function
\begin{equation}\label{e52}
\underline w(x,t):=- K h(x;\theta_0)e^{ \alpha t}-\lambda (t-t_0)^2
+\tilde \gamma(\theta_0, t_0)-3\epsilon\,, \quad (x,t)\in
Q_\delta^j:=N_\delta^j\times [ t_\delta, t_0]\,,
\end{equation}
with $N_\delta^j=C^R_{\theta_0,\delta}\cap B_j$, where
$R>\max\{\hat{R},R_\epsilon\}$ and where $K>0$, $\alpha>0$,
$\lambda>0$ are constants to be chosen later. We get
\[
a\Delta \underline w +  c \underline w \geq K e^{\alpha t}\, -\, K
c\,  h(x;\theta_0) e^{\alpha t}-\|c \|_\infty(\|\tilde \gamma
\|_\infty + \lambda T^2 + 3) \quad \textrm{in}\;\; Q^j_\delta\,.
\]
Therefore,
\begin{equation}\label{e53}
\begin{aligned}
\partial_t \underline w - a\Delta \underline w -  c \underline w - f\, \leq&\, - \alpha \, K h(x,\theta_0) e^{\alpha t}-
2 \l(t-t_0)- K e^{\alpha t}+ c K h(x,\theta_0) e^{\alpha t}\\ &\,
+\, \|c\|_\infty (\|\tilde \gamma\|_\infty +\lambda T^2 +3)+\|f
\|_\infty\leq 0 \qquad\qquad \textrm{in}\;\; N^j_\delta\times
(t_\delta, t_0)\,,
\end{aligned}
\end{equation}
if
\begin{equation}\label{e54}
\alpha\geq \|c \|_\infty\,,
\end{equation}
and
\begin{equation}\label{e55}
K\geq 2  \lambda T + \|f\|_\infty+\| c\|_\infty(\|\tilde
\gamma\|_\infty+ \lambda T^2+ 3)+\|f\|_\infty\,.
\end{equation}
Furthermore, it follows from \eqref{e51} that for $j>R$
\begin{equation}\label{e56}
\underline w(x,t) \leq u_j(x, t) = \tilde \gamma(\theta,t)\quad
\textrm{for}\;\; x=(r,\theta)\in \partial N^j_\delta\cap \partial
B_j,\, t\in (t_\delta, t_0)\,.
\end{equation}
Let $m_\epsilon=m_{\delta,R}>0$ be the constant appearing in
inequality \eqref{e25}, relative to
$\partial\mathcal{C}^R_{\theta_0,\delta}$ (recall that here
$\delta=\delta(\epsilon)$ and $R=R(\epsilon)$). From \eqref{e25} and
\eqref{e50} we can infer that
\begin{equation}\label{e57}
\underline w (x,t)\leq - K m_\epsilon + \|\tilde \gamma\|_\infty
\leq u_j(x,t)\quad \textrm{for all}\;\; x\in \partial \mathcal
C^R_{\theta_0, \delta}\cap B_j,\, t\in (t_\delta, t_0),
\end{equation}
for
\begin{equation}\label{e58}
K\geq \frac{\| \tilde \gamma\|_\infty+ K_T}{m_\epsilon}\,.
\end{equation}
Now, suppose that $t_\delta>0$. From \eqref{e50} we have that
\begin{equation}\label{e59}
\underline w(x, t_\delta) \leq u(x,  t_\delta)\,,\;\, x\in
N_\delta^j
\end{equation}
if
\begin{equation}\label{e60}
\lambda \geq \frac{\| \tilde \gamma\|_\infty + K_T}{\delta^2}\,.
\end{equation}
On the other hand, if $t_\delta=0$ (this is always the case when
$t_0=0$), then from \eqref{e47}, \eqref{e51} and \eqref{e51a} we
have that
\begin{equation}\label{e61}
\begin{aligned}
\underline
w(x,0)\,&\leq\,\tilde\gamma(\theta_0,t_0)-3\epsilon\,\leq\,\tilde\gamma(\theta_0,0)-2\epsilon\\
&\leq\, u_0(r,\theta)-\epsilon\leq u_{0,j}(r,\theta)\,=\,u_j(x,0)
\quad \textrm{for all}\;\;\,
x=(r,\theta)\in\mathcal{C}^R_{\theta_0,\delta}\cap B_j\,.
\end{aligned}
\end{equation}

Now, suppose that  \eqref{e54}, \eqref{e55}, \eqref{e58} hold;
moreover, assume \eqref{e60}, if $t_\delta>0$. From \eqref{e53},
\eqref{e56}, \eqref{e57}, and \eqref{e59} if $t_\delta>0$ or
\eqref{e61} if $t_\delta=0$, it follows that $\underline w$ is a
subsolution of problem
\begin{equation}\label{e62}
\left\{
\begin{array}{ll}
\,   \partial_t w = a\Delta w + c w + f &\textrm{in}\,\,Q^j_\delta
\\& \\
\textrm{ }w \, = u_j &\textrm{in\ \ } \partial N_\delta^j\times
(t_\delta, t_0] \,,
\\& \\
\textrm{ } w \,= u_j &\textrm{in \ \ } N_\delta^j\times
\{t_\delta\}\,.
\end{array}
\right.
\end{equation}
On the other hand, $u_j$ is a solution of the same problem. Then by the maximum principle we have
\begin{equation}\label{e63}
\underline w \leq u_j \quad \textrm{in}\;\; Q^j_\delta\,.
\end{equation}

Analogously we have that
\begin{equation}\label{e51b}
u_j \leq \overline w \quad \textrm{in}\;\; Q^j_\delta\,,
\end{equation}
where
\begin{equation}\label{e64} \overline w(x,t):=K h(x;\theta_0)
e^{\alpha t}+\lambda (t-t_0)^2+\tilde \gamma(\theta_0,
t_0)+3\epsilon\,,\;\, (x,t)\in Q^j_\delta.
\end{equation}

\medskip
Finally, from \eqref{e63} and \eqref{e51b} we have that for any
$x\equiv (r,\theta)\in N^j_\delta$ and $t\in [t_\delta, t_0]$, with
$0<\epsilon<1$, $j>R>\max\{\hat{R},R_\epsilon\}$ and
$0<\delta<\min\{\hat{\delta},\delta(\epsilon)\}$,
\begin{equation}\label{e65}
\big|u_j(x,t)-\tilde \gamma(\theta_0, t_0) \big|\, \leq\, K h(x;
\theta_0) e^{\alpha t}+ \l (t-t_0)^2+3\epsilon\,.
\end{equation}
Note that these constants depend on $\epsilon$, but do not depend on
$\theta_0\in\mathbb S^{m-1}$, $t_0\in[0,T]$. Now we pass to the
limit as $j\to \infty$ in \eqref{e65}, and choose $\theta=\theta_0$,
$t=t_0$. So, for every $r>R$,
\begin{equation}\label{e66}
 \big|  u(r,\theta_0, t_0)-\tilde \gamma(\theta_0, t_0)\big| \leq K
h(r,\theta_0;\theta_0) e^{\alpha t_0}+3\epsilon\,.\end{equation} In
view of \eqref{e66} and \eqref{e26}, we have
\[\big|u(r,\theta_0, t_0)-\tilde \gamma(\theta_0, t_0)\big| < 4\epsilon\]
for $r>0$ large enough, independent of $\theta_0\in\mathbb S^{m-1}$,
$t_0\in[0,T]$. But $\epsilon>0$ is arbitrarily small, therefore
\eqref{e4} follows.

In order to prove uniqueness, suppose by contradiction that there
exist two solutions $u_1$, $u_2$ of problem \eqref{e2} satisfying
\eqref{e4}. Then set $w:= u_1-u_2$. Take any $\epsilon>0$. In view
of \eqref{e4}, there exists $R_\epsilon>0$ such that
\begin{equation}\label{e80}
|w(x,t)|\leq \epsilon \quad \textrm{for all}\;\; x\in M\setminus
B_{R_\epsilon},\, t\in [0, T]\,.
\end{equation}
Moreover, $w$ is a subsolution of problem
\begin{equation}\label{e67}
\left\{
\begin{array}{ll}
\,\partial_t v = a\Delta v  + c v
&\textrm{in}\,\,B_{R_\epsilon}\times (0, T]
\\& \\
\textrm{ }v \, = \epsilon & \textrm{in\ \ } \partial B_{R_\epsilon}\times (0, T] \,
\\& \\
\textrm{ }v \, = 0 & \textrm{in\ \ } B_{R_\epsilon}\times \{0\} \,.
\end{array}
\right.
\end{equation}
It is easily seen that the function
$$z(x,t):= \epsilon \, e^{\|c\|_\infty t}\quad \big(x\in M, t\in [0, T]\big) $$
is a supersolution of problem \eqref{e67}. By the comparison principle,
\begin{equation}\label{e68}
w(x,t)\,\leq\,  z(t)\,\leq\,\epsilon e^{\|c\|_\infty T} \quad
\textrm{for all}\;\; x\in \overline{B}_{R_\epsilon},\, t\in [0,
T]\,.
\end{equation}
Similarly, it can be shown that
\begin{equation}\label{e69}
w(x,t)\geq - z(t) \,\geq\,-\epsilon e^{\|c\|_\infty T} \quad
\textrm{for all}\;\; x\in \overline{B}_{R_\epsilon},\, t\in [0,
T]\,.
\end{equation}
Then, from \eqref{e80}, \eqref{e68} and \eqref{e69}, we see that for
some positive constant $\Lambda$ we have
\[|w(x,y)|\leq\Lambda\epsilon\quad \textrm{for all}\;\; x\in M,\, t\in [0,
T]\,.
\]
Letting $\epsilon \to 0^+$, we get $w\equiv0$ in $M\times [0, T]\,.$
Hence the proof is complete.
\end{proof}

\section{Examples}\label{SecEXE}
\begin{exe}
Let $M=M_\psi$ be a model manifold with $\psi\in \mathcal A$, see
\eqref{e8}, and let $\psi(r)\sim e^{r^{\alpha}}$ as $r\to \infty$
for some $\alpha>0$. Note that for $\mathbb H^m$ we have $\alpha=1$.
In general, there holds
\[\int_1^{\infty}\left(\int_r^\infty \frac{d\xi}{\psi^{m-1}(\xi)}\right) \frac {\psi^{m-1}(r)}{\underline a(r)}\,dr < \infty,\]
provided that
\begin{equation}\label{e70}
\int_1^\infty \frac 1{r^{\alpha-1}\underline a(r)}\,dr <\infty\,.
\end{equation}
Hence, if \eqref{e70} holds, and
\[\underline a(r)\leq C_0  e^{2r^{\alpha}}\quad \textrm{for all}\,\, r\geq R_0\]
for some $R_0>0$, then Theorems \ref{teor1} and \ref{teor2} apply.

Moreover, note that in the special case when
$$
a(r(x))\equiv \underline a(r)\equiv 1\,,
$$
condition \eqref{e70} is satisfied for any $\alpha>2$\,. Furthermore, observe that if $\psi=e^{r^{\alpha}}$ for all $r\geq r_0$, for some $r_0>0,$
then we have
\[K_\omega(x) \sim - \alpha^2 [r(x)]^{2\alpha -2}\quad \textrm{as}\;\,\,r(x)\to \infty\,.\]
\end{exe}

\begin{exe}
Let $M$ be a Cartan-Hadamard manifold. Suppose that, for some $\alpha>0$,
\begin{equation}\label{e74}
\operatorname{K}_\omega(x)\leq -\alpha^2 \quad \textrm{for all}\;\, x\in M\,.
\end{equation}
Now note that by defining $\psi(r):=\frac{\sinh(\alpha r)}{\alpha}$
we have $\psi\in\mathcal{A}$ and
$\frac{\psi''(r)}{\psi(r)}=\alpha^2$. Let $\omega(r)$ be defined as
in \eqref{e13} and suppose that
\begin{equation}\label{e72}
\int_1^{\infty}{\omega(r)} dr <\infty\,,
\end{equation}
then {\bf(HP1)}-$(ii)$ and condition \eqref{e13} hold, if we set
$\underline{a}(r):=\frac{1}{\omega(r)}$.

Consider equation \eqref{e1} with $c=f\equiv 0$, namely equation
\begin{equation}\label{e73a}
a(x) \Delta u = 0\quad \textrm{on}\;\, M\,,
\end{equation}
which is of course equivalent to equation
\begin{equation}\label{e71}
\Delta u = 0\quad \textrm{on}\;\, M\,.
\end{equation}

Since $a(x)$ is arbitrary, we can choose $a\in
C^\sigma_{\text{loc}}(M)$ for some $\sigma>0$, with
\begin{equation}\label{e77}
a(x)\geq\underline a(r(x))=\frac 1{\omega(r(x))} \quad \textrm{for
all}\,\, x\in M\setminus B_{R_0}\,,
\end{equation}
for some $R_0>0$. Thus, Theorem \ref{teor1} applies for equation
\eqref{e73a}, and so also for equation \eqref{e71}\,. Moreover,
Theorem \ref{teor2} can be applied for the parabolic problem
\eqref{e2} with $c=f=0$ and $a\in C^\sigma_{\text{loc}}(M)$ such
that \eqref{e77} is satisfied.

As already noted in Remark \ref{rem1}, if $M\equiv M_\psi$ is a
model manifold then $\omega(r)=\frac 1{\psi^2(r)}\,.$ Clearly,
\eqref{e74} and \eqref{e72} are satisfied on $\mathbb H^m\,.$
\end{exe}

\begin{exe}
Let $M\equiv M_\phi$ be a model manifold. Suppose that
\begin{equation}\label{e75}
\operatorname{K}_\omega(x)\leq - \frac{A}{r^2 \log(r)}\quad
\textrm{for all}\,\, x\in M\setminus B_{R_0}\,,
\end{equation}
for some $A>1$, $R_0>0$. By \cite[Proposition 3.4]{Choi}, for any
$\beta\in(1,A)$ there exists some $R_1\geq R_0$ such that
\begin{equation}\label{e76}
\phi(r) \geq  \psi(r)\quad \textrm{for all}\,\, r\geq R_1\,,
\end{equation}
where $\psi(r):=r \log^{\beta}(r)$. Moreover, for some $R_2>0$ large
enough,
\begin{equation}\label{e81}
\operatorname{K}_\omega(x)\leq -\frac{\psi''(r)}{\psi(r)}\quad
\textrm{for all}\,\, x\in M\setminus B_{R_2}\,.
\end{equation}

Now choose $a\in C^\sigma_{\text{loc}}(M)$, for some $\sigma>0$,
with
\begin{equation}\label{e78}
a(x)\geq\underline a(r(x))=C_0\psi^2(r(x)) \quad \textrm{for
all}\,\, x\in M\setminus B_{R_2}\,,
\end{equation}
for some $C_0>0$. In view of \eqref{e76}, \eqref{e81} and the very
definition of $\psi$ it is easily seen that hypothesis {\bf (HP1)}
and condition \eqref{e13} are satisfied. Thus, Theorem \ref{teor1}
applies for equation \eqref{e73a}, and hence also for equation
\eqref{e71}\,. This is in accordance with \cite[Theorem 3.6]{Choi}.
Moreover, Theorem \ref{teor2} can be applied for the parabolic
problem \eqref{e3} with $c=f=0$ and $a$ defined as in \eqref{e78}.

\end{exe}

\vspace{1cm}

\noindent Paolo Mastrolia\\
Universit\`{a} degli Studi di Milano\\
Dip di Matematica\\
via Saldini 50\\
20133 Milano. \vspace{0,5cm}

\noindent Dario D. Monticelli\\
Politecnico di Milano\\
Dip. di Matematica\\
via Bonardi 9\\
20133 Milano. \vspace{0,5cm}

\noindent Fabio Punzo\\
Universit\`{a} della Calabria\\
Dip. di Matematica e Informatica\\
via Bucci, Cubo 31B\\
87036 Rende (CS).

\end{document}